\documentclass[12pt]{amsart}

\usepackage{multicol}

\usepackage{amssymb}
\usepackage{fancyhdr}
\usepackage{amsmath}
\usepackage{amsfonts}
\usepackage{mathrsfs}
\usepackage[all]{xy} % For commutative diagrams
\usepackage{color}
\usepackage[makeroom]{cancel}
\usepackage{ulem}

\newcommand{\disk}{\ensuremath{\mathbb{D}} } % unit disk
\newcommand{\sphere}{\overline{\Bbb{C}}} %Riemann sphere
\newcommand{\riem}{\Sigma}  %Riemann surface
\renewcommand{\Bbb}[1]{\ensuremath{\mathbb{#1}}}
\theoremstyle{plain}
        \newtheorem{theorem}{Theorem}[section]
        
        \newtheorem{proposition}[theorem]{Proposition}
        \newtheorem{corollary}[theorem]{Corollary}

\theoremstyle{definition}
        \newtheorem{definition}[theorem]{Definition}

\theoremstyle{remark}
    \newtheorem{remark}[theorem]{Remark}

\numberwithin{equation}{section} % Equation labels are 'section'.'eq #'
\numberwithin{figure}{section} % Figures labela are 'section.'fig #'

\usepackage{fullpage}

\title{Faber series for $L^2$ holomorphic one-forms on Riemann surfaces with boundary}

\author[E. Schippers]{Eric Schippers}

\author[M. Shirazi]{Mohammad Shirazi}

\thanks{Eric Schippers acknowledges the support of the Natural Sciences and Engineering Research Council of Canada (NSERC)}


\begin{document}

\begin{abstract}
    Consider a compact surface $\mathscr{R}$ with distinguished points $z_1,\ldots,z_n$ and conformal maps $f_k$ from the unit disk into non-overlapping quasidisks on $\mathscr{R}$ taking $0$ to $z_k$. Let $\riem$ be the Riemann surface obtained by removing the closures of the images of $f_k$ from $\mathscr{R}$. 
    We define forms which are meromorphic on $\mathscr{R}$ with poles only at $z_1,\ldots,z_n$, which we call Faber-Tietz forms. These are analogous to Faber polynomials in the sphere. We show that any $L^2$ holomorphic one-form on $\riem$ is uniquely expressible as a series of Faber-Tietz forms. This series converges both in $L^2(\riem)$ and uniformly on compact subsets of $\riem$.   
\end{abstract}

\maketitle

\begin{section}{Introduction}
 Faber series are series of polynomials associated to a Jordan domain in the sphere, which approximate holomorphic functions in $\Omega$ under sufficient assumptions on the domain and curve. Various choices of the sense of convergence, regularity of the domain, and regularity of the holomorphic function being approximated  are possible. Faber series have certain advantages over power series. For example, under many analytic choices, the Faber series converges uniformly to the function on arbitrary compact subsets of the domain in question.   For the general theory see the monographs of J. M Anderson and P. K. Suetin \cite{Anderson,Suetin}. In this paper, we give a generalization of Faber series for $L^2$ holomorphic one-forms on Riemann surfaces $\riem$ obtained by removing $n$ Jordan domains from a compact surface $\mathscr{R}$.  We call these Faber-Tietz series after H. Tietz who pioneered the theory of Faber series on Riemann surfaces. We consider Faber-Tietz series of an $L^2$ one-forms on Riemann surfaces converges in $L^2$, for Riemann surfaces obtained by removing non-overlapping quasidisks from a compact surface. 
	
 We begin by recalling some of the basic theory in the plane. 	
 A central aspect of Faber polynomials is that they are constructed from the Riemann map of the complement of the domain on which the approximation takes place.	
 Let $\mathbb{C}$ denote the complex plane, $\sphere$ the Riemann sphere, and $\disk = \{ z: |z|<1 \}$ the unit disk. 
 Let $f:\mathbb{D} \rightarrow \mathbb{C}$ be a conformal map. For integers $n \geq 1$, the Faber polynomials $\Phi_n$ associated to this map are the principal parts of $(f^{-1}(z))^{-n}$. Holomorphic functions on the complement of $f(\mathbb{D})$ in the sphere can be approximated by series in Faber polynomials, and under various analytic and geometric conditions this approximation has advantages over power series expansions. 
 
 The results of this paper are partly motivated by the following elegant theorem of Y. Shen \cite{ShenFaber}, with one direction obtained earlier by A. \c{C}avu\c{s} \cite{Cavus}. Let $U$ be the interior of the complement of $f(\disk)$ in the sphere. Every holomorphic function $h$ on $U$ with $L^2$ derivatives has a unique Faber series with derivatives convergent to $h$ in $L^2$ 
 if and only if $f(\disk)$ is a quasidisk.  (We have made some slight changes to the statement for consistency with the notation in this paper, without altering the essential content of the theorem.) This is closely related to a theorem of V. Napalkov and R. Yulmukhametov \cite{Nap_Yulm} which says the following. Let $\Omega$ be a Jordan domain in the plane and let $\mathcal{B}(\Omega)$ denote the Bergman space of $L^2$ holomorphic functions on $\Omega$ and $\overline{B(\Omega)}$, and $\overline{B(\Omega)}$ denote the anti-holomorphic $L^2$ functions.  Let $\Omega^*$ denote the exterior Jordan domain in the sphere. The integral transform from $\overline{\mathcal{B}(\Omega)} \rightarrow \mathcal{B}(\Omega^*)$
 \begin{equation} \label{eq:Hilbert_transform}
   \overline{h} \mapsto \iint_{\Omega} \frac{\overline{h(\zeta)}}{(\zeta-z)^2} dA_\zeta  \ \ z \in \Omega^* 
 \end{equation}
 is an isomorphism if and only if $\Omega$ is a quasidisk.  
 
 We generalize the result of \c{C}avu\c{s} and Shen in one direction to $L^2$ one-forms on Riemann surfaces $\mathscr{R}$ obtained by removing many quasidisks. In the case that $\mathscr{R}$ is a sphere, the one-forms reduce to $h(z)dz$ where $h$ is as above.  To do this we required a natural generalization of Faber polynomials to Riemann surfaces.
 Such generalizations were first considered by H. Tietz 
 \cite{Tietz52, Tietz53, Tietz57, Tietz572}. 
 The forms replacing the Faber polynomials are holomorphic on the compact surface $\mathscr{R}$, except for  poles at the distinguished points $z_1,\ldots,z_n$. The approach of Tietz was to define a Cauchy kernel on the Riemann surface, imposing normalizations to make it unique, and treating the Cauchy kernel as a generating function. See for example \cite{Duren_book,Pommerenkebook,Suetin} for generating functions for Faber polynomials in the sphere. Equivalently, one can treat the Cauchy kernel as defining an integral operator (the Faber operator) and obtain Faber series as the output of this operator.  This produces series which converge uniformly on compact subsets. The non-uniqueness of the Cauchy kernel for general Riemann surfaces comes from the existence of a $g$-dimensional family of holomorphic one-forms on the surface, available to add to a specific Cauchy kernel without changing its singular part. In the sphere, there are no non-trivial holomorphic one-forms so the non-uniqueness issue does not arise. 
 
 Our approach instead uses an operator of M. Schiffer \cite{Courant_Schiffer,Schiffer_Spencer,Schippers_Staubach_Plemelj} generalizing (\ref{eq:Hilbert_transform}) to Riemann surfaces. Although we do not emphasize the point in this paper, this operator is obtained by differentiating a Cauchy kernel we call the Cauchy-Royden kernel \cite{Royden}, which is in turn obtained from Green's function. The ``Schiffer operator'' has a holomorphic integral kernel, and we can obtain meromorphic forms on the compact surface with specified poles using the Schiffer operator. We call them Faber-Tietz forms, although these are not precisely the same as those defined by Tietz. We use isomorphism theorems for Schiffer and Faber operators of M. Shirazi \cite{Shirazi_thesis} (see also E. Schippers, M. Shirazi, and W. Staubach \cite{Schippers_Shirazi_Staubach}) and E. Schippers and W. Staubach \cite{Schippers_Staubach_scattering}, generalizing that of Napalkov and Yulmukhametov, in order to produce the series representations and convergence.  
 
 Our main results are as follows. Let $f_1,\ldots,f_n$ be conformal maps of the unit disk $\disk$ into a compact Riemann surface $\mathscr{R}$. Assume that $\Omega_k = f_k(\disk)$ are quasidisks with non-intersecting closure, and let $\Omega = \Omega_1 \cup \cdots \cup \Omega_n$. Let ${\riem}$ be the complement of the closure of $\Omega$ 
 in $\mathscr{R}$. We define Faber-Tietz forms, and show these are holomorphic on $\mathscr{R}$ except for specified poles at the points $z_k=f(0)$. We 
  show
 that every $L^2$ holomorphic one-form $\alpha$ on $\riem$ has a unique decomposition as a Faber series which converges in $L^2$ to $\alpha$. This series  also converges uniformly on compact subsets of $\riem$.  
 
 Here we are concerned with the development of explicit Faber-type series with a fixed basis.  
  Earlier density results were obtained by the authors with W. Staubach, for nested Riemann surfaces with boundary \cite{Schippers_Shirazi_Staubach}. It was shown that $L^2$ holomorphic one-forms on the inner surface can be approximated in $L^2$ by holomorphic one-forms on the outer surface (and similarly for Dirichlet spaces of functions). Density results for $L^2$ spaces of differentials of general orders (automorphic forms) were obtained by N. Askaripour and T. Barron \cite{Barron_Askaripour,Foth_Askaripour}, using Poincar\'e series; see also references therein for density results for $L^1$ holomorphic quadratic differentials. For other theorems involving approximation on a larger region in other senses such as uniform approximation, see e.g., Gauthier and Sharifi \cite{Gauthier_Sharifi}. For literature reviews see F. Sharifi \cite{Sharifi} and Fornæss et al. \cite{Fornaess_et_al}. 
\end{section}
\begin{section}{Preliminaries and background}
\begin{subsection}{Function spaces and Green's function of a compact surface}

\begin{definition}
 On any Riemann surface,  define the dual of the almost 
 complex structure,  $\ast$ in local coordinates $z=x+iy$,  by 
 \[  \ast (a\, dx + b \, dy) = a \,dy - b \,dx. \]
This is independent of the choice of coordinates.
It can also be computed in coordinates that for any complex function $h$ 
\begin{equation} \label{eq:Wirtinger_to_hodge}
    2 \partial_z h = dh + i \ast dh.
\end{equation}

\begin{definition}
 We say a complex-valued function $f$ on an open set $U$ is {harmonic}
 if it is $C^2$ on $U$ and $d \ast d f =0$. We say that a complex one-form $\alpha$ is harmonic if it is $C^1$ and satisfies
 both $d\alpha =0$ and $d \ast \alpha =0$. 
\end{definition}
  Equivalently, harmonic one-forms are those which can be expressed locally as $df$ for some harmonic function $f$. Harmonic one-forms and functions must of course be $C^\infty$. \\
  
   Denote complex conjugation of functions and forms with a bar, e.g. $\overline{\alpha}$.
 A holomorphic one-form is one which can be written in coordinates as $h(z)\,dz$ for a holomorphic function $h$, while an anti-holomorphic one-form is one which can be locally written $\overline{h(z)}\, d\bar{z}$ for a holomorphic function $h$.  
 
Denote by $L^2(U)$ the set of one-forms $\omega$ on an open set $U$ in a Riemann surface which satisfy
\[   \iint_U \omega \wedge \ast \overline{\omega} < \infty  \]
(observe that the integrand is positive at every point, as can be seen by writing the expression in local coordinates).  In this paper $U$ will be a compact Riemann surface without boundary or a Riemann surface with border. For $\alpha, \beta \in L^2(U)$ the pairing 
\begin{equation} \label{eq:form_inner_product}
 (\omega_1,\omega_2) =  \iint_U \omega_1 \wedge \ast \overline{\omega_2}
\end{equation}
is defined.

\begin{definition}\label{defn:bergman spaces}
The Bergman space of holomorphic one forms is 
\begin{equation}
    \mathcal{A}(U) = \{ \alpha \in L^2(U) \,:\, \alpha \ \text{holomorphic} \}.
\end{equation} 
 The Bergman space of anti-holomorphic one-forms is denoted $\overline{\mathcal{A}(U)}$.   We will also denote 
\begin{equation}
    \mathcal{A}_{\mathrm{harm}}(U) =\{ \alpha \in L^2(U) \,:\, \alpha \ \text{harmonic} \}
\end{equation}
and call this the harmonic Bergman space.
\end{definition}
For all choices of $U$ in this paper, the Bergman spaces are Hilbert spaces with respect to the inner product (\ref{eq:form_inner_product}).

Observe that $\mathcal{A}(U)$ and $\overline{\mathcal{A}(U)}$ are orthogonal with respect to the inner product \eqref{eq:form_inner_product}.  In fact we have the direct sum decomposition
\begin{equation}\label{direct sum decomposition}
 \mathcal{A}_{\mathrm{harm}}(U) = \mathcal{A}(U) \oplus \overline{\mathcal{A}(U)}.    
\end{equation}      
If we restrict the inner product to 
 $\alpha, \beta \in \mathcal{A}(U)$ then since $\ast \overline{\beta} = i \overline{\beta}$, we have   
\[  (\alpha,\beta) = i \iint_U \alpha \wedge \overline{\beta}.      \]

Let $f: U \rightarrow V$ be a biholomorphism. We denote the pull-back of $\alpha \in \mathcal{A}_{\mathrm{harm}}(V)$
under $f$ by $f^*\alpha.$ 
Explicitly, if $\alpha$ is given in local coordinates $w$ by $a(w)\, dw + \overline{b(w)} \, d\bar{w}$ and $w=f(z),$ then the pull-back is given by 
\[   f^* \left( a(w)\, dw + \overline{b(w)} \,d\bar{w} \right)= a(f(z)) f'(z)\, dz + \overline{b(f(z))} \overline{f'(z)}\, d\bar{z}.   \]

The Bergman spaces are all conformally invariant, in the sense that if $f:U \rightarrow V$ is a biholomorphism, then $f^*\mathcal{A}(V) = \mathcal{A}(U)$ and this preserves the inner product 
(i.e. $f^*$ is unitary operator).
Similar statements hold for the anti-holomorphic and full 
harmonic spaces. \\

\begin{definition}\label{def: exact holo and harm forms}
 We define the space $\mathcal{A}^e_{\mathrm{harm}}(U)$ as the subspace of exact elements of $\mathcal{A}_{\mathrm{harm}}(U)$, and similarly for $\mathcal{A}^\mathrm{e}(\riem)$ and $\overline{\mathcal{A}^\mathrm{e}(\riem)}$.  
\end{definition}

We also consider one-forms which have zero boundary periods,
which we call semi-exact.  
\begin{definition} \label{de:semi_exact}
 Let $\riem$ be a Riemann surface with $n$ borders homeomorphic to the circle. 
 We say that an $L^2$ one-form $\alpha \in \mathcal{A}_{\mathrm{harm}}(\riem)$ is semi-exact if for any simple closed curve $\gamma$ homotopic to a boundary curve $\partial_k \riem$, 
 \[ \int_{\gamma}  \alpha =0. \]
 The class of semi-exact holomorphic one-forms on $\riem$ is denoted $\mathcal{A}^{\mathrm{se}}(\riem)$.
\end{definition}

 For compact surfaces $\mathscr{R}$, one defines {Green's function} $\mathscr{G}$ as the unique function
 $\mathscr{G}(w,w_0;z,q)$ satisfying, for $z,q$ distinct,
 \begin{enumerate}
  \item $\mathscr{G}$ is harmonic in $w$ on $\mathscr{R} \backslash \{z,q\}$;
  \item for a local coordinate $\phi$ on an open set $U$ containing $z$, $\mathscr{G}(w,w_0;z,q) + \log| \phi(w) -\phi(z) |$ is harmonic 
   for $w \in U$;
  \item for a local coordinate $\phi$ on an open set $U$ containing $q$, $\mathscr{G}(w,w_0;z,q) - \log| \phi(w) -\phi(q) |$ is harmonic 
   for $w \in U$;
  \item $\mathscr{G}(w_0,w_0;z,q)=0$  
  for all $w_0 \neq q,z$.
 \end{enumerate}
\end{definition}

The existence of such a function is a standard fact about Riemann surfaces, see for example H. Royden \cite{Royden}. The term ``Green's function'', which is not in universal use, is taken from \cite{Royden}.
 It follows from the properties that $\mathscr{G}$ is also harmonic in $z$ anywhere that it is non-singular. We will usually leave out the point ``$w_0$'' in the notation for $\mathscr{G}$ because it only changes an additive constant that vanishes after differentiation.
\end{subsection}
\begin{subsection}{Capped surfaces and Schiffer operators} 
 In this paper, by a conformal map we mean a holomorphic bijection.  If we say that $f$ is a conformal map of a domain $D$ without specifying the image, we mean that it is a holomorphic bijection onto its image; that is, it is one-to-one on $D.$

 Given a Riemann surface $R$, we say that a simple closed curve $\Gamma$ in $R$ is a quasicircle if there is a conformal map $\phi:U \rightarrow \mathbb{C}$ from an open set $U$ containing $\Gamma$ such that $\phi(\Gamma)$ is a quasicircle in the plane, 
 in the usual sense.

 By a capped surface $\mathscr{R},\Omega_1,\ldots,\Omega_n$ we mean the following data: 
 \begin{enumerate}
     \item a compact surface $\mathscr{R}$ without boundary;
     \item a finite collection of simply-connected domains $\Omega_1,\ldots,\Omega_n$ in $\mathscr{R}$;  such that
     \begin{enumerate}
     \item the boundary of $\Omega_k$ is a Jordan curve for $k=1,\ldots,n$; and 
     \item $\text{cl} \,\Omega_k \cap \text{cl} \,\Omega_l$ is empty for all $k\neq l$. 
     \end{enumerate}
 \end{enumerate}
 The domains $\Omega_k$ are referred to as ``caps''. 
 By a capped surface with conformal maps $f_1,\ldots,f_n$ we mean a capped surface as above, together with conformal maps
 \[  f_k:\mathbb{D} \rightarrow \Omega_k. \]
 The maps $f_k$ specify distinguished points $p_k=f_k(0)$ for $k=1,\ldots,n$. 
 
 Throughout the paper we will denote 
 \[  \Omega = \cup_{k=1}^n \Omega_k   \]
 and
 \[  \riem = \mathscr{R} \backslash \mathrm{cl} \,  \Omega  \]
 where $\mathrm{cl}$ denotes topological closure, 
 and the conformal maps as 
 \[  f=(f_1,\ldots,f_n). \]
 For brevity we refer to the data for a capped surface as $(\mathscr{R},\Omega)$ and the data of a capped surface with conformal maps as $(\mathscr{R},\Omega,f)$.  
 
We define the Schiffer operator as follows:  
\begin{align*}
  \mathbf{T}:\overline{\mathcal{A}(\Omega)} & \rightarrow \mathcal{A}(\riem) \\
  \overline{\alpha} & \mapsto \frac{1}{\pi i} \iint_{\Omega} \partial_w \partial_z \mathscr{G}(w;z,q) \wedge_w \overline{\alpha(w)}.  
\end{align*} 
Here $z \in \riem$.  
This operator is bounded and independent of $q$ \cite{Schippers_Shirazi_Staubach,Schippers_Staubach_scattering}.
\end{subsection} 
\end{section}
\begin{section}{Faber series for forms}

\begin{subsection}{Faber operator for capped surfaces}
 
 Let $\mathscr{R}$ be a compact surface. Fix distinct points $z_1,\ldots,z_n \in \mathscr{R}$ and denote $\mathscr{R}'= \mathscr{R} \backslash \{z_1,\ldots,z_n\}$.
   Give $\mathscr{R}$ a homotopy basis of curves
 \[  a_1,\ldots,a_g, b_1,\ldots,b_g \]
 where $g$ is the genus of $\mathscr{R}$.  If desired, it can be chosen so that the intersection matrix
 is of the form 
 \begin{equation} \label{eq:standard_intersection}
 a_j \cdot b_k = \delta_{jk} 
 \end{equation}
 where $\delta_{jk}$ is the Kronecker delta function \cite{Farkas_Kra}.

 Every cohomology class on $\mathscr{R}$ has a harmonic representative.  That is, for any choice of periods there is a harmonic one-form $\alpha + \overline{\beta} \in \mathcal{A}(\mathscr{R}) \oplus \overline{\mathcal{A}(\mathscr{R})}$ with those periods. That is, given $\lambda_1,\ldots,\lambda_g,\mu_1,\ldots,\mu_g$ there is an  $\alpha + \overline{\beta}$ such that 
 \[ \int_{a_j} (\alpha + \overline{\beta}) = \lambda_j, \ \ \ \mathrm{and} \ \ \ \int_{b_k} (\alpha + \overline{\beta}) = \mu_k. \]
 For the rest of this paper, we fix a basis for the space of holomorphic one-forms
 \[ \mathcal{A}(\mathscr{R}) = \mathrm{span} \{ \gamma_1,\ldots, \gamma_g \}. \]
 We then have a corresponding basis for the anti-holomorphic one-forms
  \[ \overline{\mathcal{A}(\mathscr{R})} = \mathrm{span} \{ \overline{\gamma_1},\ldots, \overline{\gamma_g} \}. \]  

 \begin{remark} \label{re:normalized_a_periods} A particular basis for $\mathcal{A}_{\mathrm{harm}}(\mathscr{R})$ is the following. If the homology basis satisfies the intersection property \eqref{eq:standard_intersection} then it is well known that there are unique one-forms $\gamma_k \in \mathcal{A}(\mathscr{R})$ such that 
\[
\int_{a_j} \gamma_k = \delta_{jk}. 
\] 
The period matrix is given by 
\[ \Pi_{jk}=  \int_{b_j} \gamma_k.  \]
which is symmetric and has positive-definite imaginary part.
\end{remark}

 For $k=1,\ldots,n-1,$ we let $\beta_k$
 be the unique meromorphic one-form on $\mathscr{R}$ with a simple pole of residue $1$ at $z_k$, a simple pole of residue $-1$ at $z_n$, and no other poles.  Let $\Gamma_1,\ldots,\Gamma_n$ be a collection of simple closed curves with the property that $\Gamma_k$ has winding number $1$ with respect to 
 $z_k$
 and $0$ with respect to the other points. If desired, we may choose them to be non-intersecting curves. For example, given a capped surface $\mathscr{R}$ the curves $\partial \Omega_k$, oriented positively with respect to $\Omega_k$ form such a collection of curves.
Given any meromorphic one-form $\alpha$ on $\mathscr{R}$ which is holomorphic on $\mathscr{R} \backslash \{ z_1,\ldots,z_n \}$
there is a unique $\beta \in \mathrm{span} \{ \beta_1,\ldots,\beta_{n-1} \}$ such that 
 \[  
 \int_{\Gamma_k} \alpha - \beta = 0 \ \ \mathrm{for} \ \ k=1,\ldots,n.  
 \]
 
 Define the space 
 \[ 
 \overline{X}(\riem) = \left\{ \alpha \in \mathcal{A}(\riem) : \exists \,  \overline{\sigma} \in \overline{\mathcal{A}(\mathscr{R})} \ \mathrm{ s.t. } \ \alpha - \overline{\sigma} \in \mathcal{A}^e(\riem)
 \right\}.      
 \]
The following was proven in \cite{Schippers_Staubach_scattering}. 
\begin{theorem} \label{co:Tonetwo_iso_full}  Let $(\mathscr{R},\Omega)$ be a capped surface, and $\riem = \mathscr{R} \backslash \mathrm{cl} \, \Omega$.   If the caps are quasidisks, then 
$\mathbf{T}$ is an isomorphism onto $\overline{X}(\riem)$.
 \end{theorem}
 
Let 
   \begin{align*}
       \mathbf{R}_{\riem}:\mathcal{A}(\mathscr{R}) & \rightarrow \mathcal{A}(\riem) \\
       \mathbf{R}_{\Omega}:\mathcal{A}(\mathscr{R}) & \rightarrow \mathcal{A}(\Omega)
   \end{align*}
   be the restriction operators.
   Similarly $\overline{\mathbf{R}_{\riem}}$, $\overline{\mathbf{R}_{\Omega}}$ are the restriction operators on $\overline{\mathcal{A}(\mathscr{R})}$ while $\mathbf{R}^h_{\Omega}$, $\mathbf{R}^h_{\riem}$ denote the restriction operators on $\mathcal{A}_{\mathrm{harm}}(\mathscr{R})$. 
   
   The following generalization of the Faber operator plays a central role in the approximation theory. 
   Define
   \begin{align*}
       \Theta: \overline{\mathcal{A}(\Omega)} \oplus \mathcal{A}(\mathscr{R}) & \rightarrow \mathcal{A}^{se}(\riem)  \\
       (\overline{\gamma},\tau) & \mapsto \mathbf{T} \overline{\gamma} +  {\mathbf{R}_\riem} \tau.  
   \end{align*}
    
   We have the following result.
   \begin{theorem}[\cite{Schippers_Staubach_scattering}] \label{th:big_isomorpism}  Let $(\mathscr{R},\Omega)$ be a capped surface, and $\riem = \mathscr{R} \backslash \mathrm{cl} \, \Omega$.   If the caps are quasidisks, then $\Theta$ is a bounded isomorphism. 
   \end{theorem}
\begin{remark} Theorems \ref{co:Tonetwo_iso_full} and \ref{th:big_isomorpism} can be seen as generalizations of the theorem of Napalkov and Yulmukhametov \cite{Nap_Yulm} mentioned in the introduction, in one direction. Their theorem says that in the case that $\mathscr{R} = \sphere$, $\mathbf{T}$ is an isomorphism if and only if $\Omega$ is a quasidisk.  These theorems take into account the cohomology of the surface in two different ways.
\end{remark}
  
  \begin{remark}
   The inverse of this isomorphism can be given explicitly \cite{Schippers_Staubach_scattering}.  For an exact harmonic form $\alpha$ on $\riem$, we let $\mathbf{O}^e \alpha \in \mathcal{A}_{\mathrm{harm}}(\Omega)$ be the unique harmonic one-form whose primitive  has the same boundary values as the primitive of $\alpha$ on $\Gamma$, up to a constant. What is meant by ``the same boundary values'' requires substantial clarification. It can be shown that this procedure is well-defined for quasicircles - for details see \cite{Schippers_Staubach_Plemelj,Schippers_Staubach_scattering}. Taking this for granted, we can give a simple formula for the inverse.   
   Then, we define  
    \begin{align*}
       {\mathbf{O}}': \mathcal{A}^{se}_{\mathrm{harm}}(\riem) & \rightarrow \mathcal{A}_{\mathrm{harm}}(\Omega)  \\
      \beta & \mapsto \mathbf{O}^e(\beta -  {\mathbf{R}}_\riem^h \sigma) +
        {\mathbf{R}}_\Omega^h\sigma  
    \end{align*}
    where $\sigma$ is the unique element of 
    $\mathcal{A}_{\mathrm{harm}}(\mathscr{R})$ such that $\beta-  {\mathbf{R}}_\riem^h \sigma$ is exact.  
 Using this, we define the augmented overfare map
 \begin{align}   \label{eq:Oaugmented_definition}  
   {\mathbf{O}}^{aug}: \mathcal{A}^{se}_{\mathrm{harm}}(\riem) & \rightarrow \mathcal{A}_{\mathrm{harm}}(\Omega) \oplus \mathcal{A}_{\mathrm{harm}}(\mathscr{R})  \nonumber \\
   \beta   & \mapsto (  {\mathbf{O}}' \beta, \sigma ).
 \end{align} 
   Let 
   $\mathbf{P}_{\mathscr{R}}$ be the bounded orthogonal projection  
   \[   {\mathbf{P}}_{\mathscr{R}}: \mathcal{A}_{\mathrm{harm}}(\mathscr{R}) \rightarrow \mathcal{A}(\mathscr{R})   \] 
   and similarly ${\overline{\mathbf{P}}}_{\mathscr{R}}$ is the projection onto the anti-holomorphic part of $\mathcal{A}_{\mathrm{harm}}(\Omega)$. 
 
  With these definitions, 
  %\[   \Theta^{-1}= \overline{\mathbf{P}}_{cap} \mathbf{O}^{aug}.  \]
  \[  \Theta^{-1} = \left(-\overline{\mathbf{P}}_{\Omega}\right) \oplus  {\mathbf{P}}_{\mathscr{R}} \]
  (Note that there is a change in sign convention from \cite{Schippers_Staubach_scattering}).
  \end{remark}
\end{subsection}
\end{section}
\begin{section}{Faber-Tietz series}
\begin{subsection}{Faber-Tietz forms} 
 
 Let $\mathscr{R},\Omega_1,\ldots,\Omega_n$ be a capped surface with conformal maps $f_1,\ldots,f_n,$ 
 \[
 f_k:\mathbb{D} \rightarrow \Omega_k. 
 \]
 Denote $\Omega = \Omega_1 \cup \cdots \cup \Omega_n$. Assume that $z_k=f_k(0)$ for $k=1,\ldots,n$.

 We define the promised Faber-Tietz forms. These are meromorphic differentials with possible singularities only  at points $z_1,\ldots,z_n$. The differentials depend on $f_1,\ldots,f_n$. 
 
 Let $e^m_k \in \overline{\mathcal{A}(\disk)} \oplus  \cdots \oplus \overline{\mathcal{A}(\disk)}$ (where there are $n$ copies of $\overline{\mathcal{A}(\disk)}$ on the right hand side) be the one-form
 \[  \left( 0,\ldots, 0, m\, \overline{z}^{m-1} d\, \bar{z}, 0, \ldots,0 \right) \]
 where the non-zero entry is in the $k$ component of the direct sum. 
 \begin{definition}[Faber-Tietz forms]  Let $m \geq 1$ be an integer and $k \in \{1,\ldots,n \}$. The $m,k$-th Faber-Tietz form with singularity at $z_k$ is 
 \[  \alpha^m_k = \mathbf{T} (f^{-1})^* (e^m_k).  \]
 \end{definition}
 Here, the pull-back $(f^{-1})^* e$ of $e \in \overline{\mathcal{A}(\disk)} \oplus  \cdots \oplus \overline{\mathcal{A}(\disk)}$ refers to the one-form in $\overline{\mathcal{A}(\Omega)}$ whose restriction to $\Omega_k$ is $(f_k^{-1})^* e$.  By properties of $\mathbf{T}$ (see Theorem \ref{co:Tonetwo_iso_full}), $\alpha^m_k$ is in $\mathcal{A}(\riem)$.  More can be said: it extends to a meromorphic differential on $\mathscr{R}$ with poles only at the points $z_1,\ldots,z_n$. We show this now. 

 \begin{theorem} \label{th:Faber-Tietz_have_specified_pole} For any integer $m \geq 1$ and $k \in \{1,\ldots,n \}$, 
  the Faber-Tietz form $\alpha_k^m$ extends holomorphically to $\mathscr{R} \backslash \{z_k\}$.  Furthermore, $\alpha^m_k$ has a pole of order $m+1$ at $z_k$; in coordinates defined by $f_k$, the principal part at $z_k$ is given by 
  \[  f^* \alpha^m_k = 
  \left(\frac{m}{\zeta^{m+1}} + h(\zeta )\right) d\zeta \]
  where $h(\zeta)$ is holomorphic at $0$.  
 \end{theorem}
 \begin{proof}  Fix $k \in \{1,\ldots,n\}$ and $m \geq 1$. 
 %Set $h(z)=z^{-m}$, and let $H(z)=\overline{h(1/\bar{z})}$. In that case $H$ is a (holomorphic) polynomial in $1/z$ and 
 For $0<r<1$ let $\Gamma^k_r = f_k(|z|=r)$ and note that this curve is analytic and homotopic to $\partial_k \Omega$. 
 For any annulus $\mathbb{A}_r = \{z: r<|z|<1 \}$  we set $A_r=f(\mathbb{A}_r)$. Fix $q \in \riem$ (the location is not important).  Fix an $r_0$ such that $0<r_0<1$.  
 
 %by conformal invariance of CNT boundary values we have that 
 %\[  \mathbf{G}_{A_r,\Omega_k} H \circ f^{-1} = \overline{h} \circ f^{-1}.    \]
 
   Using \cite[Theorem 4.8]{Schippers_Staubach_Plemelj} in the third equality below, we compute using Stokes' theorem that 
  \begin{align} \label{eq:integral_expression_singularity_proof} 
      \alpha^m_k = \mathbf{T} (f^{-1})^* (e^m_k)(z) & = -  \frac{1}{\pi i} \iint_{{\Omega_k}} \partial_z \partial_w \mathscr{G}(w;z,q) \wedge_w (f^{-1})^*e^m_k \nonumber \\& = - \lim_{\epsilon \searrow 0} \frac{1}{\pi i} \int_{\Gamma^k_{r}} \partial_z \partial_w \mathscr{G}(w;z,q)   \overline{(f^{-1}(w))}^{m}  \nonumber \\
      & = - \lim_{\epsilon \searrow 0} \frac{1}{\pi i} \int_{\Gamma^k_r} \partial_z \partial_w \mathscr{G}(w;z,q) (f^{-1}(w))^{-m} \nonumber \\
      & =  - \frac{1}{\pi i} \int_{\Gamma^k_{r_0}}  \partial_z  \partial_w \mathscr{G}(w;z,q) (f^{-1}(w))^{-m}  \nonumber \\
      & = - \frac{1}{\pi i} \partial_z \int_{\Gamma^k_{r_0}}  \partial_w \mathscr{G}(w;z,q) (f^{-1}(w))^{-m}.
  \end{align}
  In the second-to-last equality, we have used the fact that the integrand is holomorphic in $w$ in $A_r$.  
  
  The far right hand side in the above expression is holomorphic in $z$ on the exterior of $\Gamma^k_{r_0}$. Since $r_0$ is arbitrary, we see that  $\alpha^m_k$
  in fact holomorphic on $\mathscr{R} \backslash \{z_k\}$. 
  
  To prove the second claim, we express the final integral in \eqref{eq:integral_expression_singularity_proof} in the coordinate induced by $f$. Let $\eta$, $\zeta$ be given by $w =f(\eta)$ and $z = f(\zeta)$.  Assuming that $|\zeta|>r=|\eta|$, using the defining properties of $\mathscr{G}$ with $f$ in place of the local coordinate $\phi$ near $z$ we see that 
  \begin{align*}
       f^* \alpha^m_k(z) & = - f^* \frac{1}{\pi i} \partial_{z} \int_{\Gamma^k_{r}} \partial_w \mathscr{G}(w;z,q) (f^{-1}(w))^{-m} \\
       & = \partial_{\zeta} \int_{|\eta|=r} \eta^{-m} \left(  \frac{1}{2 \pi i} \frac{1}{\eta - \zeta} + \psi(\eta;\zeta) \right) d\eta  \\
       & = \partial_{\zeta}\left( - \frac{1}{\zeta^m} + \int_{|\eta|=r} \eta^{-m} \psi(\eta;\zeta) d\eta \right) \\
       & = 
       \frac{m }{\zeta^{m+1}}d\zeta
       + \int_{|\eta|=r}  \eta^{-m} \partial_\zeta \psi(\eta;\zeta) d\eta 
  \end{align*}
  where $\psi(\eta;\zeta)$ is harmonic and non-singular in $\zeta$, and holomorphic in $\eta$. The integral over $|z|=r$ is assumed to be in the direction determined by the positive orientation with respect to $0$ (counter-clockwise in the plane).
  Since the second integrand is holomorphic in $\eta$, the second integral is independent of $r$. Letting $r \searrow 0$ we see that the above formula holds for $\zeta$ on a punctured neighbourhood of $0$. Thus, 
  \[  f^*\alpha^m_k(\zeta) =  \left(\frac{m}{\zeta^{m+1}} + h(\zeta) \right) d\zeta \]
  where the second term is holomorphic and non-singular in $\zeta$.  
 \end{proof}

 Henceforth, by $\alpha^m_k$ we mean the meromorphic extension to $\mathscr{R}$.  
 
\end{subsection}
\begin{subsection}{Faber series of one-forms}
 In this section we define the Faber-Tietz series and show that if the caps are quasidisks, then the Faber-Tietz series of any $L^2$ holomorphic one-form has a unique Faber series converging in $L^2$.
 
 It was shown in \cite[Corollary 7.5]{Schippers_Staubach_scattering} that,  
 every cohomology class on $\riem$ has a representative in  $\mathcal{A}(\riem)$. We can therefore make the following statement. 
 \begin{proposition} \label{pr:remove_some_cohomology} Let $(\mathscr{R},\Omega,f)$ be a capped surface with conformal maps and $\riem = \mathscr{R} \backslash \mathrm{cl} \, \Omega$.  
  Let $\nu \in \mathcal{A}(\riem)$. There are unique one-forms 
  \[  \tau = \sum_{k=1}^{g}c_k \gamma_k \]
  and 
  \[
  \beta = \sum_{k=1}^{n-1} \epsilon_k \beta_k.   
  \] 
  
  such that 
  \[  \nu - \beta-\tau  \in \overline{X}(\riem). \]
 \end{proposition}
 \begin{proof}
  Given $\nu$, set 
  \[  \epsilon_k = \int_{\partial_k \riem} \nu. \]
  By the integral over $\partial_k \riem$, we mean the integral over any smooth curve $\gamma$ homotopic to $\partial_k \riem$. This is independent of $\gamma$ so it is well-defined.
  If we set  
  \[ \beta = \sum_{k=1}^n \epsilon_k \beta_k \]
  then the boundary periods of $\nu-\beta$ are zero. There 
  is a unique  harmonic one-form $\rho \in \mathcal{A}_{\mathrm{harm}}(\mathscr{R})$ such that
  \[  \int_{a_k} (\nu - \beta)  = \int_{a_k}  \rho \ \ \text{and} \ \  \int_{b_k} (\nu - \beta) = \int_{b_k} \rho, \ \ \ k =1,\ldots,g, \]
  so that
  \[   \nu - \beta - \rho \in \mathcal{A}^e(\riem).  \]
  Using the decomposition 
  \[  \rho = \tau + \overline{\sigma}, \ \ \ \delta,\sigma \in \mathcal{A}(\mathscr{R}),    \]
  we obtain that $\nu - \tau - \beta \in \overline{X}(\riem)$ as claimed.
 \end{proof}
\begin{remark}
 The coefficients $\epsilon_k$ and $c_k$ can be computed using integral formulas. The first formula  
 \begin{equation} \label{eq:boundary_Faber_coefficients}
    \epsilon_k   = \int_{\partial_k \riem} \nu
 \end{equation}
 arose already in the proof above.  An explicit expression for $c_k$ can be obtained given a specific choice of basis for $\mathcal{A}_{\mathrm{harm}}(\mathscr{R})$. For example, if 
 the basis is as in Remark \ref{re:normalized_a_periods}, we solve the equations
 \begin{align*}
     A_j & = \int_{a_j} \sum_k \left(c_k \gamma_k + d_k \overline{\gamma_k} \right)  =   c_j + d_j \\
     B_j & =   \int_{b_j} \sum_k \left(c_k \gamma_k + d_k \overline{\gamma_k} \right)   =   \sum_k \left( \Pi_{jk} c_k + \overline{\Pi}_{jk} d_j \right) 
 \end{align*}
 for $c_k$ and $d_k$, where 
 \[   A_j = \int_{a_j} (\nu-\beta),  \  \ \ \ \ B_j = \int_{b_j} (\nu - \beta).      \]
 Setting
 \[ \tau=\sum_{k=1}^n c_k \gamma_k \]
 it is easily seen that 
 $\nu-\beta-\tau \in \overline{X}(\riem)$. 
 
 \end{remark} 
 
 \vspace{0.8cm}

 We can now define the Faber multi-series as follows.
  Assume that the caps are quasidisks.  Let $\nu \in \mathcal{A}(\riem)$, and let 
  \[  \beta= \sum_{k=1}^n \epsilon_k \beta_k  \]
  be as in Proposition \ref{pr:remove_some_cohomology}. Define  $\overline{\delta}$ and $\tau$ by
  \[  
  (\overline{\delta},\tau) = \Theta^{-1} ( \nu-\beta)  
  \]
  and observe that
  \[  \tau = \sum_{k=1}^g c_k \gamma_k.  \]
  for some $c_k$.  Finally set 
  \[  \overline{\delta}_k = \left. \overline{\delta} \right|_{\Omega_k}   \]
  for $k=1,\ldots,n$.  Now we can expand $f_k^*\overline{\delta_k}$ in a power series in $\bar{z}$
  \begin{equation} \label{eq:Faber_coefficients}
   f_k^*\overline{\delta_k} = \sum_{m=1}^\infty m h^k_m   \overline{z}^{m-1} d\bar{z}.
  \end{equation} 
  
  \begin{definition} \label{de:Faber_series}  Let $(\mathscr{R},\Omega,f)$ be a capped surface with conformal maps and $\riem = \mathscr{R} \backslash \mathrm{cl} \, \Omega$.  Assume that the caps are quasidisks.
   With notation as above, the Faber series of $\nu \in \mathcal{A}(\riem)$ is 
   \begin{equation} \label{eq:Faber_series}
     \sum_{k=1}^n \epsilon_k \beta_k +  \sum_{k=1}^g c_k \gamma_k  + \sum_{m=1}^\infty \sum_{k=1}^n  h_m^k \alpha^m_k.    
   \end{equation}
  \end{definition}
  
  It follows from the definition of the Faber-Tietz forms and Theorem \ref{th:big_isomorpism} that
  \begin{theorem} \label{th:Faber_series_forms} Let $(\mathscr{R},\Omega,f)$ be a capped surface with conformal maps and $\riem = \mathscr{R} \backslash \mathrm{cl} \, \Omega$. Assume that the caps are quasidisks.
  	 
   Let $\nu \in \mathcal{A}(\riem)$. The Faber series of Definition \ref{de:Faber_series} converges in $\mathcal{A}(\riem)$ to $\nu$. The Faber series is the unique series of the form \eqref{eq:Faber_series} converging to $\nu$ in $\mathcal{A}(\riem)$.    
  \end{theorem}
  \begin{remark}
   Note that every term in the Faber series is a meromorphic differential in $\mathcal{A}(\mathscr{R})$ with poles only at the fixed points $z_1,\ldots,z_n$.  
  \end{remark}
  
  Furthermore, it converges uniformly on compact sets, in the following sense.  
  \begin{theorem} \label{th:uniform_conv_forms}
   Let $\zeta=\phi(z)$ be any coordinate chart on $\riem$, taking $U$ into the complex plane, so that in coordinates and $\nu=G(\zeta)d\zeta$ and the 
   $M$th
   partial sum is 
   \[   g_M(\zeta)d\zeta = \sum_{k=1}^n \epsilon_k \beta_k +  \sum_{k=1}^g c_k \gamma_k  + \sum_{m=1}^M \sum_{k=1}^n  h_m^k \alpha^m_k.  \]   
   
   For any compact set $K \subseteq U,$ $g_M \rightarrow G$  uniformly on $K$ as $M \rightarrow \infty$. 
  \end{theorem}
  \begin{proof}
   The sequence $g_M(\zeta)d\zeta$ converges to $G(\zeta) d\zeta$ in the Bergman space $\mathcal{A}(\phi(U))$, where $\phi(U)$ is a subset of the complex plane. By a standard result for Bergman spaces, there is a constant $C_K$ such that 
   \[   \sup_{\zeta \in K} |h(\zeta)| \leq C_k \| h(\zeta) d\zeta \|_{\mathcal{A}(U)}   \]
   Applying this to $h(\zeta)=g_M(\zeta) - G(\zeta)$ proves the claim. 
  \end{proof}
\end{subsection}
\begin{subsection}{Conformal invariance of Faber-Tietz series}

In this section we observe that the Faber series is invariant under conformal maps $g:\mathscr{R}_1 \rightarrow \mathscr{R}_2$.  This is an automatic consequence of the construction. Letting $\mathcal{G}_k$ and $\mathbf{T}_k$ denote Green's function and the Schiffer operator on $\mathscr{R}_k$ for $k=1,2$, we have by uniqueness $\mathcal{G}_1 = \mathcal{G}_2 \circ (g \times g)$. Thus the operators $\mathbf{T}_k$ are invariant in the sense that  $g^* \mathbf{T}_2 = \mathbf{T}_1 g^*$ \cite{Schippers_Staubach_scattering}. Furthermore the function spaces are conformally invariant, e.g. $g^*\mathcal{A}(\riem_2) = \mathcal{A}(\riem_1)$.  We then have 
\begin{theorem}[conformal invariance of Faber-Tietz forms] \label{th:conformal_invariance_Faber_Tietz_forms}
	Let $\mathscr{R}$, $\Omega=\Omega_1 ,\ldots,\Omega_n$ be a capped surface with conformal maps $f=(f_1,\ldots,f_n)$. Let $g:\mathscr{R} \rightarrow \hat{\mathscr{R}}$ be a conformal map. Setting $\hat{\Omega}_k=g_k(\Omega_k)$ and $\hat{f}_k=f_k \circ g$, we have that $\hat{\mathscr{R}}$, $\hat{\Omega}_1,\ldots,\hat{\Omega}_n$ is a capped surface with conformal maps $\hat{f}_1,\ldots,\hat{f}_n$.  
	
	If $\hat{\alpha}_k^m$ are the Faber-Tietz forms associated to  $(\hat{\mathscr{R}}, \hat{\Omega},\hat{f})$, then the Faber-Tietz forms associated to $(\mathscr{R},\Omega,f)$ are 
	\[ \alpha^m_k = g^* \hat{\alpha}^m_k. \]
\end{theorem}
It immediately follows that the Faber-Tietz series are also conformally invariant.
\begin{theorem} \label{th:conformal_invariance_Faber_Tietz_series}
	Let $(\mathscr{R},\Omega,f)$ and $(\hat{\mathscr{R}},\hat{\Omega},\hat{f})$ be capped surfaces with conformal maps, and $g:\mathscr{R} \rightarrow \hat{\mathscr{R}}$ be a conformal bijection. Assume that the caps $\Omega$ are quasidisks (and thus so are $\hat{\Omega}$).  Let $\riem = \mathscr{R} \backslash \mbox{cl}\, \Omega$ and similarly for $\hat{\riem}$. Let $\hat{\beta}_k$, $k=1,\ldots,n$ and  $\hat{\gamma}_k$, $k=1,\ldots,g$ be as in Definition \ref{de:Faber_series} and set 
	\[ \beta_k = g^* \hat{\beta}_k,\ \ \  k =1,\ldots,n \ \ \ \mbox{and} \ \ \  \gamma_k = g^* \hat{\gamma}_k,\ \ \  k =1,\ldots,g.  \]
	
	 If the Faber-Tietz series of $\hat{\nu} \in \mathcal{A}(\riem)$ is 
	\[  \hat{\nu} =  \sum_{k=1}^n \epsilon_k \hat{\beta}_k +  \sum_{k=1}^g c_k \hat{\gamma}_k  + \sum_{m=1}^\infty \sum_{k=1}^n  h_m^k \alpha^m_k  \] 
	then the Faber-Tietz series of $\nu = g^* \hat{\nu}$ is  
	\[    \sum_{k=1}^n \epsilon_k \beta_k +  \sum_{k=1}^g c_k \gamma_k  + \sum_{m=1}^\infty \sum_{k=1}^n  h_m^k \alpha^m_k.      \]
\end{theorem}
This is obtained by pulling back the Faber-Tietz series term by term under $g$, using Theorem \ref{th:conformal_invariance_Faber_Tietz_forms} together with the fact that pull-back is an isometry. 

\end{subsection}
\begin{subsection}{Recovery of the Faber series on the sphere}
 In this section we show how to recover the Faber series on the sphere $\mathscr{R} = \sphere$.  See e.g. \cite{Anderson,Duren_book,Pommerenkebook,Suetin} for the standard definition on the sphere.  In particular
 we obtain the convergence theorem of \c{C}avu\c{s} \cite{Cavus} and Shen \cite{ShenFaber} (in one direction).  
 
 We fix a single conformal map $f:\mathbb{D} \rightarrow \Omega$, though the same procedure works for arbitrarily many non-overlapping quasidisks.   
 Observe that there are no non-trivial holomorphic one-forms on the sphere, that is $\mathcal{A}(\mathscr{R})$ consists only of the form which is identically zero. So $\Theta$ is just $\mathbf{T}$ up to the trivial second component, and we can use $\mathbf{T}$ in place of $\Theta$ in the construction of the Faber-Tietz series.  Furthermore $\riem = \mathscr{R} \backslash \mathrm{cl} \,\Omega$ is simply connected, so $\overline{X}(\riem) = \mathcal{A}(\riem)$ and Theorem \ref{co:Tonetwo_iso_full} says that $\mathbf{T}$ is an isomorphism. 
 
 Assume that $\Omega$ is a bounded quasidisk, so that $\infty$ is in the interior of $\riem$. It is easily checked that 
 \[ \mathscr{G}(w,w_0;z,q) = \log \left| \frac{(z-w_0)(w-q)}{(w-z)(q-w_0)} \right|   \]
 so 
 \[ - \frac{1}{\pi i} \partial_w \mathscr{G}(w,w_0;z,q) = \frac{1}{2 \pi i} \left( \frac{1}{w-z} - \frac{1}{w-q} \right) dw. \]  Setting $q=\infty$ we obtain the Cauchy kernel
 \begin{equation*} 
  - \frac{1}{\pi i} \partial_w \mathscr{G}(w,w_0;z,\infty) = \frac{1}{2 \pi i}  \frac{dw}{w-z}
 \end{equation*}
 and the kernel of the Schiffer operator $\mathbf{T}$ is  
 \begin{equation} \label{eq:Schiffer_on_sphere}
   \frac{1}{\pi i} \partial_z \partial_w \mathscr{G}(w,w_0;z,\infty) = - \frac{1}{ \pi }  \frac{dw \,dz}{(w-z)^2}.
 \end{equation}
 The Schiffer operator takes the form, for $\alpha = \overline{h(z)} d\bar{z} \in \overline{\mathcal{A}(\Omega)}$, 
 \[ \left[ \mathbf{T} \alpha \right](z) =  \frac{1}{\pi} \iint_{\Omega} \frac{\overline{h(w)}}{(w-z)^2} \frac{d\bar{w} \wedge dw}{2i} dz.  \]
 Thus  one direction of Napalkov and Yulmukhametov's theorem \cite{Nap_Yulm} mentioned in the introduction is a special case of Theorem \ref{th:big_isomorpism}.  
 \begin{corollary}
 	Let $\riem$ be a quasidisk in the sphere and $\Omega$ be the complement of its closure. Then  $\mathbf{T}:\overline{\mathcal{A}(\Omega)} \rightarrow \mathcal{A}(\riem)$ is an isomorphism.
 \end{corollary}
In fact they show that, among Jordan domains, $\riem$ is a quasidisk if and only if $\mathbf{T}$ is an isomorphism. 
 
 Denote $e^m= m \bar{z}^{m-1}d\bar{z}$ for $m \geq 1$. 
 The Faber-Tietz forms are by definition 
 \[  
 \alpha^m = \mathbf{T} (f^{-1})^* e^m. 
 \]
 By Theorem \ref{th:Faber-Tietz_have_specified_pole}, $\alpha^m(z) = p(z) dz$ where $p \in \mathbb{C}[1/z]$ is a polynomial of degree $m-1$. 
 Let $\Phi^m$ denote the unique primitives of $\alpha^m$ in $\mathbb{C}[1/z]$ vanishing at 
 $q=\infty$.  
 
 We claim that for fixed $r_0 \in (0,1)$ 
 \begin{equation} \label{eq:equiv_to_Faber_poly_definition}
 \Phi^m(z) = \frac{1}{2 \pi i} \int_{f(|z|=r_0)} \frac{(f^{-1}(w))^m}{w-z} dw. 
 \end{equation} 
  This  is just the standard definition of the Faber polynomials associated to $f$ and $\riem$ \cite{Duren_book,Pommerenkebook,Suetin}.  To prove the claim, we compute
  \begin{align*}
   \partial_z \Phi^m & = \frac{1}{2 \pi i} \int_{f(|z|=r_0)} \frac{(f^{-1})(w)^{-m}}{(w-z)^2} dw \,dz\\
   & =   \lim_{r \nearrow 1} \frac{1}{2 \pi i}  \int_{f(|z|=r)} \frac{(f^{-1})(w)^{-m}}{(w-z)^2} dw \,dz \\
   & =   \lim_{r \nearrow 1} \frac{1}{2 \pi i}  \int_{f(|z|=r)} \frac{\overline{f^{-1}(w)}^{m}}{(w-z)^2} dw \,dz \\
   & = \frac{1}{\pi} \iint_{\Omega} \frac{m \overline{f^{-1}(w)}^{m}}{(w-z)^2} \frac{d\bar{w} \wedge dw}{2i} \, dz 
  \end{align*}
  so 
  \begin{equation}  \label{eq:Faber_and_derivative}
   \partial_z \Phi^m = [\mathbf{T} (f^{-1})^* e^m](z).  
  \end{equation}
  where in the third equality we have applied the ``anchor lemma'' \cite[Theorem 4.8]{Schippers_Staubach_Plemelj}.   
 
  Thus in this case Theorem \ref{th:Faber_series_forms} reduces to the result of \c{C}avu\c{s} \cite{Cavus} mentioned in the introduction. 
  \begin{corollary} \label{co:Faber_series_sphere_unique_exists}
  	Let $f:\disk \rightarrow \sphere$ be a conformal map onto a quasidisk $\Omega$ and let $\riem = \sphere \backslash \mbox{cl} \, \Omega$. For any $\alpha \in \mathcal{A}(\riem)$, there is a Faber series 
  	\[  \alpha = \sum_{m=1}^\infty  h_m \alpha^m   \]
  	converging to $\alpha$ in $\mathcal{A}(\riem)$. This is the unique Faber series converging to $\alpha \in \mathcal{A}(\riem)$.
  \end{corollary} 
  Here   the $h_m$ are determined by 
  \begin{equation*}
  \sum_{m=1}^\infty h_m \alpha^m 
   = \mathbf{T} (f^{-1})^* \left( \sum_{m=1}^\infty h_m e^m \right). 
  \end{equation*} 
  where 
  \[  \sum_{m=1}^\infty h_m e^m \]
   is just the power series in $\mathbb{D}$ for the unique pre-image of $\alpha$ under $\mathbf{T}(f^{-1})^*$.
   
   \c{C}avu\c{s} phrases this result in terms of functions, which in the terminology here are precisely the coefficients $h(z)$ of one-forms $\alpha(z) = h(z) dz$.  
   By equation (\ref{eq:Faber_and_derivative}), $e^m =(\Phi^m)'(z) dz$. \c{C}avu\c{s} refers to the series in terms of $(\Phi^m)'$ as generalized Faber series.   
   
   \begin{remark} Shen \cite{ShenFaber} shows that the converse holds: if any $\alpha \in \mathcal{A}(\riem)$ has such a unique expansion, then $\riem$ is a quasidisk. Equivalently, there is an isomorphism between the class of square summable sequences $\ell^2$ and $\mathcal{A}(\mathbb{D})$.  This also relates to the characterization of quasicircles in terms of the Grunsky matrix of Ch. Pommerenke and R. K\"uhnau \cite{Pommerenkebook}. 
   \end{remark}
   
  \begin{remark} 
  	 \c{C}avu\c{s} and Shen consider Faber series on an bounded domain and choose the conformal map $f$ to map onto the unbounded complement.  We have altered the conventions here for consistency. 
  \end{remark} 
  
  We can also phrase this more classically as a Faber series of functions in terms of $\Phi^m$.  Let $\mathcal{D}_\infty(\riem)$ denote the Dirichlet space of holomorphic functions vanishing at $\infty$, with norm $\| h \|_{\mathcal{D}_*(\riem)}= \| \partial h \|_{\mathcal{A}(\riem)}$. Similarly $\overline{\mathcal{D}_0(\disk)}$ is the Dirichlet space of anti-holomorphic functions on $\disk$ vanishing at $0$.   
  It follows immediately from Corollary \ref{co:Faber_series_sphere_unique_exists} and equation (\ref{eq:Faber_and_derivative}) that for any $H(z) \in \mathcal{D}_\infty(\riem)$,  there is a unique expansion 
 \[   H(z) = \sum_{m=1}^\infty h_m \Phi^m   \]
 which is convergent in $\mathcal{D}_\infty(\riem)$.  
 
 For more on Faber series for functions with $L^2$ derivatives on quasidisks in the case of the sphere, and the interrelation between the theorems of \c{C}avu\c{s}, Shen, and Napalkov-Yulmukhametov, see the survey \cite{Schippers_Staubach_Grunsky_expository}. 
\end{subsection}
\end{section}

\end{document}